\documentclass[11pt]{amsart}
\usepackage{graphicx}

\usepackage{lmodern}			
\usepackage[T1]{fontenc}		
\usepackage[utf8]{inputenc}		
\usepackage{lastpage}			
\usepackage{indentfirst}		
\usepackage{color,xcolor}		
\usepackage{graphicx}			
\usepackage{microtype} 			
\usepackage{dynkin-diagrams}
\usepackage[displaymath, mathlines]{lineno}

\usepackage[subentrycounter,seeautonumberlist,nonumberlist=true]{glossaries}

\usepackage{lipsum}			
\usepackage[portuguese,onelanguage]{algorithm2e}	
 \usepackage{amsbsy}			
 \usepackage{amscd}			
 \usepackage{amsfonts}			
 \usepackage{amsmath}			
 \usepackage{amssymb}			
 \usepackage{amstext}			
\usepackage{amsthm}			
\usepackage{hyperref}			
\usepackage{cleveref}			
\usepackage{dsfont}			
 \usepackage{ifthen}			
\usepackage{listings}           	
 \usepackage{lscape}             	
 \usepackage{mathrsfs}			
 \usepackage{pdfpages}           	

 \usepackage{verbatim}

\usepackage{latexsym}
\usepackage{amsthm}
\usepackage{times}
\usepackage{graphicx}
\usepackage{graphics}
\usepackage{hyperref}         
\usepackage{fancybox}
\usepackage{fancyhdr}
\usepackage{ae}
\usepackage[all]{xy}

\newcommand{\R}{\mathbb{R}}

\newcommand{\lie}{\mathfrak{g}}
\newcommand{\Ad}{\text{Ad}}

\newcommand{\m}{\mathfrak{m}}

\newcommand{\h}{\mathfrak{h}}

\numberwithin{equation}{section}
\newtheorem{theorem}{Theorem}[section]

\newtheorem{lemma}[theorem]{Lemma}
\newtheorem{corollary}[theorem]{Corollary}
\theoremstyle{definition}
\newtheorem{definition}[theorem]{Definition}
\newtheorem{example}[theorem]{Example}
\theoremstyle{remark}
\newtheorem{remark}[theorem]{Remark}

\begin{document}

\title[$\lambda_1$-extremality of KÄhler-Einstein Metrics]{Riemannian $\lambda_1$-extremal Metrics on Generalized Flag Manifolds}

\author{Kennerson N. S. Lima}
\address{Unidade Acadêmica de Matemática, Universidade Federal de Campina Grande, 
Paraíba, PB}
\email{kennerson@mat.ufcg.edu.br}

\begin{abstract}
In this work, we will establish new classification results concerning $\lambda_1$-extremality for partial flag manifolds using a sufficient and necessary condition, in terms of Lie theoretic data, for a K\"ahler-Einstein metric over a generalized flag manifold to be a critical point for the functional that assigns for each Riemannian metric its first positive eigenvalue of the associated Laplacian. 
\end{abstract}

\keywords{Extremal Metric; First Positive Eigenvalue; Kähler Invariant Metric; Flag manifolds.}
\maketitle

\section{Introduction}
The study of the Laplacian spectrum acting on smooth functions defined on a Riemannian manifold is widely explored with particular focuse on its first eigenvalue. Many recent works have pointed out the importance of critical points of the functional that, for each metric in some appropriate space of Riemannian metrics, assigns its first eigenvalue of the Hodge Laplacian operator. In \cite{Soufi}, ILIAS, S. and SOUFI, A. E. present the following concept of extremal metrics for the first eigenvalue functional.

Let $M$ be a compact manifold and be $\mathcal{R}$ the set of all Riemannian metrics on $M$ with fixed volume. The functional $\lambda_1:\mathcal{R}\longrightarrow\mathbb{R}$ that assigns for each $g\in\mathcal{R}$ the first eigenvalue $\lambda_1(g)$ of the Laplacian $\Delta_g=-\text{div}_g\circ\nabla_g$ is continuous. Moreover, if $g_t$ is an analytic curve in $\mathcal{R}$, the function $\lambda(g_t)$ has left and right derivatives w.r.t. $t$.

A metric $g\in\mathcal{R}$ is said to be $\lambda_1$-extremal if for each analytic curve $g_t$ in $\mathcal{R}$, $t\in (-\epsilon,\epsilon)$, with $g_0=g$, we have $$\frac{d}{dt}|_{t=0^+}\lambda_1(g_t)\leq 0\leq \frac{d}{dt}|_{t=0^-}\lambda_1(g_t).$$ 


The abovementioned extremality property guarantees that the manifold with a $\lambda_1$-extremal metric can be minimally and isometrically immersed in a sphere whose radius depends on the first eigenvalue of the extremal metric. This statement is also proved in the work \cite{Soufi} of El Sofi and Ilias.

Let us consider a {\it generalized flag manifold} $G/K$ endowed with a complex structure $J$, where $G$ is a compact connected semisimple Lie group and $K$ is a compact subgroup which coincides with the centralizer in $G$ of a torus.

The set of invariant Kähler metrics on $M=G/K$ (with respect to $J$) with the same volume of the Kähler-Einstein metric $\bar{g}$ can be nicely parametrized in the open {\it Weyl chamber} whose construction we outline in \ref{Section4} of this paper. In the context of this article, the functional given by the first positive eigenvalue will be considered over such an open set.

Panelli and Podestà in \cite{Panelli} obtained a sufficient and necessary condition for the Kähler-Einstein metric $\bar{g}$, defined over a generalized flag manifold $M =G/ K$, to be $\lambda_1$-extremal. It is worth noting that this criterion depends only on some data of the Lie algebra $\mathfrak{g}$ of $G$. This is given in terms of a suitable restriction of the functional, which is the sum of complementary simple roots of $\mathfrak{g}$.

Consequently, in the same paper, we have the following characterization for {\it maximal flag manifolds}.

\begin{theorem} [\cite{Panelli}] Let $G$ be a compact classic Lie group. If $(G/T,J)$ is a maximal flag manifold $(\textrm{dim T}\geq2)$, $\bar{g}$ is $\emph{$\lambda_1^{IK}$-extremal}$ if and only if $G=SU(3)$.
\end{theorem}

We will apply the condition mentioned above in \cite{Panelli} to classify families of {\it partial} flag manifolds according to their number of isotropic summands. Furthermore, for the special of partial flag manifolds with three isotropic summands given by $SU(l+m+n)/S(U(l)\times U(m)\times U(n))$, $l,m,n\in\mathbb{Z},$ $l,m,n\geq 1$, we will establish a sufficient and necessary condition for its Kähler-Einstein metric to be $\lambda_1$-extremal in terms of $l,m,n$.

The case of partial flag manifolds imposes the difficulty of treating each of their families in a particular way according to the dimensions and number of their respective isotropic summands. In contrast, in the maximal ones, these data have the same description for all families. Such Lie theoretic data are fundamental in the construction and application of the criteria we use here.

\section{First Eigenvalue and Extremal Metrics}

\label{Section1}


Berger, in his work \cite{Berger}, gives the following descrption for the lateral derivatives above mentioned in our Introduction. 

\begin{lemma} \emph{(\cite{Berger})} Given a metric $\bar{g}$ with first positive eigenvalue $\lambda_1(\bar{g})$ of multiplicity $m$, for each analytic curve $g_t$ with $t\in (-\epsilon,\epsilon)$ and $g_0=\bar{g}$, there are $L^2(g_t)$-orthogonal functions $u^1_t, \ldots, u^m_t\in C^{\infty}(M)$ and $\Lambda^1_t, \ldots,\Lambda^m_t \in \R$, depending analytically on $t$, such that for all $j=1,\ldots,m$ one has $\Lambda^j_0=\lambda_1(\bar{g})$ and $$\Delta_{g_t}u^j_t=\Lambda^j_t\cdot u^j_t, \ \ \forall \ \ t\in (-\epsilon,\epsilon).$$
\end{lemma}

From the above lemma, for $t$ sufficiently small, we have $\lambda_1(g_t)=\displaystyle\min_{1\leq j\leq m}\{\Lambda^j_t\}$ and 
$$\frac{d}{dt}|_{t=0^+}\lambda_1(g_t)=\displaystyle\min_{1\leq j\leq m}\left\{\frac{d}{dt}|_{t=0}\Lambda^j_t\right\}, \ \  \frac{d}{dt}|_{t=0^-}\lambda_1(g_t)=\displaystyle\max_{1\leq j\leq m}\left\{\frac{d}{dt}|_{t=0}\Lambda^j_t\right\}.$$ 

\begin{lemma} \emph{(\cite{Berger})} Considering the above notations, one has $$\frac{d}{dt}|_{t=0}\Lambda^j_t=-\int_M\left\langle q(u_j),h\right\rangle d\mu_{\bar{g}},$$ where $h$ is the symmetric tensor given by $\frac{d}{dt}|_{t=0}g_t$, $u_j=u^j_0$ and for all $u\in C^{\infty}(M)$ $$q(u)=du\otimes du + \frac{1}{4}\Delta_{\bar{g}}u^2\cdot \bar{g}.$$
\end{lemma}

We thus have the following expressions 

$$\frac{d}{dt}|_{t=0^+}\lambda_1(g_t)=-\displaystyle\max_{1\leq j\leq m}\left\{\int_M\left\langle q(u_j),h\right\rangle d\mu_{\bar{g}}\right\},$$ $$\frac{d}{dt}|_{t=0^-}\lambda_1(g_t)=-\displaystyle\min_{1\leq j\leq m}\left\{\int_M\left\langle q(u_j),h\right\rangle d\mu_{\bar{g}}\right\}.$$

\section{Generalized Flag Manifolds}\label{flag}
	
%
%
	
\begin{definition}	
Let $G$ be a compact semisimple Lie group with Lie algebra $\lie$. A \emph{generalized flag manifold} is the adjoint orbit of a regular element in the Lie algebra $\lie$. A generalized flag manifold is a homogeneous space of the form $G/C(T)$, where $T$ is a torus in $G$ and $C(T)$ is the centralizer of $T$ in $G$. When $T$ is maximal, it is said that $G/C(T)=G/T$ is a \emph{maximal flag manifold}. 
\end{definition}

\begin{example} (Flag manifolds of a classic Lie Group)

Type {\bf A:} $SU(n)/S(U(n_1)\times\ldots\times U(n_s)\times U(1)^m)$, 

$n=n_1+\ldots+n_s+m, n_1\geq n_2\geq\ldots\geq n_s > 1, s, m\geq 0 $

Type {\bf B:} $SO(2n+1)/U(n_1)\times\ldots\times U(n_s)\times SO(2t+1)\times U(1)^m$

Type {\bf C:} $Sp(n)/U(n_1)\times\ldots\times U(n_s)\times Sp(t)\times U(1)^m$

Type {\bf D:} $SO(2n)/U(n_1)\times\ldots\times U(n_s)\times Sp(2t)\times U(1)^m$

$n=n_1+\ldots+n_s+m+t, n_1\geq n_2\geq\ldots\geq n_s > 1, s, m, t\geq 0, t\neq 1 $

\end{example}




Let us describe a generalized flag manifold $G/K$ in terms of the Lie algebra $\lie$ root spaces of $G$. Let $\lie^{\mathbb{C}}$ be a complex semisimple Lie algebra. Given a Cartan subalgebra $\mathfrak{h}^{\mathbb{C}}$ of $\lie^{\mathbb{C}}$, denote by $\Pi$ the set of $roots$ of $\lie^{\mathbb{C}}$ with elation to $\mathfrak{h}^{\mathbb{C}}$. Consider the decomposition $$\lie^{\mathbb{C}}=\mathfrak{h}^{\mathbb{C}}\oplus\sum_{\alpha\in \Pi}\lie_{\alpha}^{\mathbb{C}},$$ where $\lie_{\alpha}^{\mathbb{C}}=\{X\in\lie^{\mathbb{C}};\forall H\in \mathfrak{h}^{\mathbb{C}}, \left[X,H\right]=\alpha(H)X\}$. Let $\Pi^+\subset \Pi$ be a choice of positive roots, $\Sigma$ the correspondent system of simple roots, and $\Theta$ a subset of $\Sigma$. Will be denoted by $\left\langle \Theta\right\rangle$ the span of $\Theta$, $\Pi_M=\Pi\setminus \left\langle \Theta\right\rangle$ the set of the {\it complementary roots} and by $\Pi_{M^+}$ the set of positive complementary roots.

A Lie subalgebra $\mathfrak{p}$ of $\lie^{\mathbb{C}}$ is called {\it parabolic} if it contains a Borel subalgebra of $\lie^{\mathbb{C}}$ (i.e., a maximal solvable Lie subalgebra of $\lie^{\mathbb{C}}$). Take $$\mathfrak{p}_{\Theta}=\mathfrak{h}^{\mathbb{C}}\oplus\sum_{\alpha\in \left\langle \Theta\right\rangle^+}\lie_{\alpha}^{\mathbb{C}}\oplus\sum_{\alpha\in \left\langle \Theta\right\rangle^+}\lie_{-\alpha}^{\mathbb{C}}\oplus\sum_{\beta\in \Pi_{M}^+}\lie_{\beta}^{\mathbb{C}}$$ the canonical parabolic subalgebra of $\lie^{\mathbb{C}}$ determined by $\Theta$ which contains the Borel subalgebra $\mathfrak{b}=\mathfrak{h}^{\mathbb{C}}\oplus\displaystyle\sum_{\beta\in \Pi_+}\lie_{\beta}^{\mathbb{C}}$.

The generalized flag manifold $\mathbb{F}_{\Theta}$ associated with $\lie^{\mathbb{C}}$ is defined as the homogeneous space $$\mathbb{F}_{\Theta}=G^{\mathbb{C}}/P_{\Theta},$$where $G^{\mathbb{C}}$ is the complex compact connected simple Lie group with Lie algebra $\lie$ and $P_{\Theta}=\{g\in G^{\mathbb{C}}; \Ad(g)\mathfrak{p}_{\Theta}=\mathfrak{p}_{\Theta}\}$ is the normalizer of $\mathfrak{p}_{\Theta}$ in $G^{\mathbb{C}}$. 

Let $G$ be the real compact form of $G^{\mathbb{C}}$ corresponding to $\lie$, i.e, $G$ is the connected Lie Group with Lie algebra $\lie$, the compact real form of $\lie^{\mathbb{C}}$. The subgroup $K_{\Theta}=G\cap P_{\Theta}$ is the centralizer of a torus. Furthermore, $G$ acts transitively on $\mathbb{F}_{\Theta}$. Since $G$ is compact, we have that $\mathbb{F}_{\Theta}$ is a compact homogeneous space, that is, $$\mathbb{F}_{\Theta}=G^{\mathbb{C}}/P_{\Theta}=G/G\cap P_{\Theta}=G/K_{\Theta},$$ accordingly characterization given above from adjoint orbits of a regular element of $\lie$. 

There are two classes of generalized flag manifolds. The first occurs when $\Theta=\emptyset$. Therefore, the parabolic subalgebra is given by  $\mathfrak{p}_{\Theta}=\mathfrak{h}^{\mathbb{C}}\oplus\displaystyle\sum_{\beta\in \Pi_+}\lie_{\beta}^{\mathbb{C}}$, that is, is equal to the Borel subalgebra of $\lie^{\mathbb{C}}$ and $T=P_{\Theta}\cap G$ is a maximal torus. In this case, $\mathbb{F}_{\Theta}=G^{\mathbb{C}}/P_{\Theta}=G/G\cap P_{\Theta}=G/T$ is called {\it maximal} flag manifold. When $\Theta\neq\emptyset$, $\mathbb{F}_{\Theta}$ is called {\it generalized} flag manifold. 

We now consider a {\it Weyl basis} for $\lie^{\mathbb{C}}$ given by $\{X_{\alpha}\}_{\alpha\in R}\cup\{H_{\alpha}\}_{\alpha\in \Sigma}\subset \lie^{\mathbb{C}}, \ \ X_{\alpha}\in \mathfrak{g}_{\alpha}^{\mathbb{C}}$. From this basis we determine a basis for $\lie$, the compact real form of $\lie^{\mathbb{C}}$, putting $$\lie=\textrm{span}_{\R}\{\sqrt{-1}H_{\alpha}, A_{\alpha}, S_{\alpha}\},$$ with $A_{\alpha}=X_{\alpha}-X_{-\alpha}$, $S_{\alpha}=\sqrt{-1}(X_{\alpha}+X_{-\alpha})$ ($A_{\alpha}=S_{\alpha}=0$ if $\alpha\notin R$) and $H_{\alpha}\in \mathfrak{h}^{\mathbb{C}}$ is such that $\alpha(\cdot)=\left\langle H_{\alpha},\cdot\right\rangle$, $\alpha\in R$. The {\it structure constants} of this basis are determined by the following relations

$$\left\{
  \begin{array}{rcl} 
  \left[A_{\alpha},S_{\beta}\right]&=& m_{\alpha,\beta}A_{\alpha+\beta}+m_{-\alpha,\beta}A_{\alpha-\beta}\\
  \left[S_{\alpha},S_{\beta}\right]&=& -m_{\alpha,\beta}A_{\alpha+\beta}-m_{\alpha,-\beta}A_{\alpha-\beta}\\
  \left[A_{\alpha},S_{\beta}\right]&=& m_{\alpha,\beta}S_{\alpha+\beta}+m_{\alpha,-\beta}S_{\alpha-\beta}\\
  \end{array}
  \right.,$$
	
	$$\left\{
  \begin{array}{rcl}  
  \left[\sqrt{-1}H_{\alpha},A_{\beta}\right]&=& \beta(H_{\alpha})S_{\beta}\\
 \left[\sqrt{-1}H_{\alpha},S_{\beta}\right]&=& -\beta(H_{\alpha})A_{\beta}\\
  \left[A_{\alpha},S_{\alpha}\right]&=& 2\sqrt{-1}H_{\alpha}\\
  \end{array}
  \right.,$$
	where $m_{\alpha,\beta}$ is such that $\left[X_{\alpha},X_{\beta}\right]=m_{\alpha,\beta}X_{\alpha+\beta}$, with $m_{\alpha,\beta}=0$ if $\alpha+\beta\notin R$ and $m_{\alpha,\beta}=-m_{-\alpha,-\beta}$. We remark that this basis is $-B$-orthogonal and $-B(A_{\alpha},A_{\alpha})=-B(S_{\alpha},S_{\alpha})=2$, where $B$ Cartan-Killing form of $\lie^{\mathbb{C}}$ (the Cartan-Killing form of $\lie^{\mathbb{C}}$ restricted to $\lie$ coincides with the Cartan-Killing form of its compact form $\lie$ and $\h_{\R}=\textrm{span}_{\R}\{\sqrt{-1}H_{\alpha}\}_{\alpha\in R}$ is a Cartan subalgebra of $\lie$). Moreover, if $\mathfrak{q}$ and $\mathfrak{s}$ are the subspaces spanned by $\{H_{\alpha},A_{\alpha}\}$ and $\{S_{\alpha}\}$ respectively, one has $$\left[\mathfrak{q},\mathfrak{q}\right]\subset\mathfrak{q}, \ \  \left[\mathfrak{q},\mathfrak{s}\right]\subset\mathfrak{s}, \ \ \left[\mathfrak{s},\mathfrak{s}\right]\subset\mathfrak{q}.$$
	
The above construction of the Weyl basis can be found, e.g., in \cite{SM}, p. 334. This will be the basis we use when dealing with flag manifolds in this work.

We will also consider a {\it generalized flag manifold} $G/K$ endowed with a complex structure $J$.

\subsection{Isotropy Representation}  

From the notion of {\it isotropy representation}, we can determine invariant metrics on certain homogeneous spaces as follows.
 
Let $G\times M\longrightarrow M$ be a differentiable and transitive action of a Lie group $G$ on the homogeneous space $(M,m)$ endowed with a $G$-invariant metric $m$. Given $x\in G$, let $K=G_x$ be the isotropy subgroup of $x$. The {\it isotropy representation} of $K$ is the homomorphism $g\in K\mapsto dg_x \in \textrm{Gl}(T_xM)$. Note that $m_{x_0}$ is a inner product on $T_{x_0}(G/K)$, invariant by such representation, where $x_0=1\cdot K$.

A homogeneous space $G/K$ is {\it reducible} if $G$ has a Lie algebra $\lie$ such that $$\lie=\mathfrak{k}\oplus\mathfrak{m},$$ with $\Ad(K)\mathfrak{m}\subset \mathfrak{m}$ and $\mathfrak{k}$ Lie algebra of $K$. If $K$ is compact, this decomposition always exists. Indeed, if we take $\mathfrak{m}=\mathfrak{k}^{\bot}$, $-B$-orthogonal complement to $\mathfrak{k}$ in $\lie$, where $B$ is the Cartan-Killing form of $\lie$.

If $G/K$ is reducible the isotropy representation of $K$ is equivalent to $\Ad |_K$, the restriction of the adjoint representation of $G$ to $K$, that is,  $$j(k)=\Ad(k)|_{\mathfrak{m}}, \forall k\in K.$$

A representation of a compact Lie group $K$ is always orthogonal (preserves inner product) on the representation space. We can conclude that every reductive homogeneous space $G/K$ has a $G$-invariant metric since such a metric is completely determined by an inner product on the tangent space at the origin $T_{x_0}(G/K)$. In fact, 
for each $k\in K$, if

\begin{itemize}
	\item $\mathcal{P}: \mathfrak{m}\longrightarrow T_{x_0}(G/K), X \mapsto \mathcal{P}(X)=\widetilde{X}(x_0)=\frac{d}{dt}(\exp(tX))x_0|_{t=0},$
	\item $\Ad(k): \mathfrak{m}\longrightarrow \mathfrak{m}, \Ad(k)X=d(C_k)_1X, (C_k:K\longrightarrow K, C_k(h)=khk^{-1})$
	\item $j(k): T_{x_0}(G/K)\longrightarrow T_{x_0}(G/K), j(k)\widetilde{X}=dk_{x_0}\widetilde{X},$

\end{itemize}
 we have that the following diagram commutes
  $$\xymatrix{\mathfrak{m}\ar[r]^{\Ad(k)}\ar[d]_{\mathcal{P}} & \mathfrak{m} \ar[d]^{\mathcal{P}}\\
              T_{x_0}(G/K)\ar[r]_{j(k)} & T_{x_0}(G/K) \\}$$
being $\mathcal{P}$ a linear isomorphism.

We remark that the set of all $G$-invariant metrics on $G/K$ are in 1-1 correspondence the set of inner products $\left\langle, \right\rangle$ on $\mathfrak{m}$, invariant by $\Ad(k)$ on $\mathfrak{m}$, for each $k\in K$, that is, $$\left\langle \Ad(k)X, \Ad(k)Y \right\rangle=\left\langle X, Y \right\rangle, \forall X,Y\in\mathfrak{m}, k\in K.$$

In the case of generalized flag manifolds the isotropy representation of $K$ leaves $\m$ invariant, i.e., $\Ad(K)\mathfrak{m}\subset \mathfrak{m}$ and decomposes it into irreducible submodules $$\m=\m_1\oplus\m_2\ldots\oplus\m_n,$$ and these submodules are inequivalent to each other. The submodules $\m_i$ are called {\it isotropy summands}. 

It follows that a $G$-invariant metric $g$ on $G/K$ is represented by a inner product $$g_{1\cdot K}=t_1Q|_{\m_1}+t_2Q|_{\m_2}+\ldots +t_nQ|_{\m_n}$$on $\m$, with $t_i$ positive constants and $Q$ is (the extension of) a inner product on $\m$, $\Ad(K)$-invariant. 

In particular, if $Q=(-B)$, with $B$ Cartan-Killing form of $G$ and $t_i=1$ for all $1\leq i\leq n$ above, the $G$-invariant metric $g$ on $G/K$ represented by the inner product $$g_{1\cdot K}=(-B)|_{\m_1}+(-B)|_{\m_2}+\ldots +(-B)|_{\m_n}$$on $\m$ is called {\it normal metric}.

\begin{remark} Let $M=G/T$ be a maximal flag manifold of a compact simple Lie group $G$. Then, the isotropy representation of $M$ decomposes into a discrete sum of $2$-dimensional pairwise non-equivalent irreducible $T$-submodules $\m_{\alpha}$ as follows: $$\m=\sum_{\alpha\in R^+}\m_{\alpha}.$$The number of these submodules is equal to the cardinality of $R^+$, the set of positive roots of $\mathfrak{g}$.
\end{remark} 

\begin{table}[htb]
\caption{The number of isotropy summands for maximal flag manifolds $G/T$}
\centering
\begin{tabular}{lcl}
\hline
Maximal flag manifold $G/T$ & Number of roots $|R|$ &$\m=\oplus_{j=1}^l$\\ 
\hline 
\vspace{0.3cm}
$SU(n+1)/T^{n},n\geq 1$&$n(n+1)$ & $l=n(n+1)/2$\\[1pt] \vspace{0.3cm}
$SO(2n+1)/T^n,n\geq 2$&$2n^2$ & $l=n^2$\\ \vspace{0.3cm} 
$Sp(n)/T^n,n\geq 3$&$2n^2$ & $l=n^2$\\ \vspace{0.3cm}
$SO(2n)/T^n,n\geq 4$&$2n(n-1)$ & $l=n(n-1)$\\ \vspace{0.3cm} 
$G_2/T$&$12$ & $l=6$\\ \vspace{0.3cm}
\end{tabular}
\end{table}
\subsection*{Flag Manifolds with 2 Isotropic Summands}

When the isotropy representation decomposes into a sum of two pairwise non-equivalent irreducible submodules $\m=\m_1\oplus\m_2$, we have the following.

In this case, $\Sigma-\Theta=\{\alpha_{i_0}\}$, where $\alpha_{i_0}$ is a simple root of height 2.

\begin{theorem}\label{2sum}\emph{(\cite{Arvanit2sum})} Let $G/K$ be a flag manifold such that $\m$ has 2 isotropic summands. Then, $G/K$ is locally isomorphic to one of the following spaces:

\begin{enumerate}

\item $E_6/SU(6)\times U(1)$;
\item $E_6/SU(2)\times SU(5)\times U(1)$;
\item $E_7/SU(7)\times U(1)$;
\item $E_7/SU(2)\times SO(10)\times U(1)$;
\item $E_7/SO(12)\times U(1)$;
\item $E_8/E_7\times U(1)$;
\item $E_8/SO(14)\times U(1)$;
\item $F_4/Sp(3)\times U(1)$;
\item $F_4/SO(7)\times U(1)$;
\item $G_2/U(2)$; ($U(2)$ represented by the short root of $G_2$)
\item $Sp(l)/U(i)\times Sp(l-i) $; ($l>0, i\neq l$)
\item $SO(2l)/U(i)\times SO(2(l-i))$; ($l>0, i\neq 1, i\neq l$)
\item $SO(2l+1)/U(i)\times SO(2(l-i)+1)$; ($l>0, i\neq 1$)
\end{enumerate}

\end{theorem}	

\subsection*{Flag Manifolds with 3 Isotropic Summands}

When the isotropy representation decomposes into a sum of three pairwise non-equivalent irreducible submodules $\m=\m_1\oplus\m_2\oplus\m_3$, we have the following possibilities for the set $\Sigma-\Theta$ (\cite{Kimura}):

\begin{itemize} 
\item [Case 1:] $\Sigma-\Theta=\{\alpha_{i_0}\}$, where $\alpha_{i_0}$ is a simple root of height 3;
\item [Case 2:]$\Sigma-\Theta=\{\alpha_{i_0},\alpha_{j_0}\}$, where $\alpha_{i_0}$ and $\alpha_{j_0}$ are simple roots of height 1.   
\end{itemize}

\begin{theorem}\label{3sum}\emph{(\cite{Kimura})}Let $G/K$ be a flag manifold such that $\m$ has 3 isotropic summands. Then, $G/K$ is locally isomorphic to one of the following spaces:

\begin{enumerate}

\item $E_6/U(2)\times SU(3)\times SU(3)$;
\item $E_7/U(3)\times SU(5)$;
\item $E_7/U(2)\times SU(6)$;
\item $E_8/U(2)\times E_6$;
\item $E_8/U(8)$;
\item $F_4/U(2)\times SU(3)$;
\item $G_2/U(2)$;
\item $SU(l+m+n)/S(U(l)\times U(m)\times U(n))$;
\item $SO(2l)/U(1)\times U(l-1)$;
\item $E_6/U(1)\times U(1)\times SO(8)$;
\end{enumerate}
where $l,m,n > 0$ in \emph{(8)} and $l\geq 4$ in \emph{(9)}. The flag manifolds in items \emph{(1)} to \emph{(7)} fall into case 1, while the others fall into case 2.

\end{theorem}	

As for determining the irreducible submodules of $\m$, we have the following description (see \cite{Kimura}):

Case 1: $\Sigma-\Theta=\{\alpha_{i_0}\}, ht(\alpha_{i_0})=3$

Define $\Delta(k)=\{\Sigma_i^n m_i\alpha_i\in \Pi^+;m_{i_0}=k\}$, for each integer $1\leq k\leq 3$. The irreducible submodules of $\m$ have the form $$\m_k=\displaystyle\Sigma_{\alpha\in\Delta(k)}u_{\alpha}, u_{\alpha} =\text{span}_{\R}\{X_{\alpha}-X_{-\alpha}, i(X_{\alpha}+X_{-\alpha})\},$$
where $\{X_{\alpha}\}_{\alpha\in\Pi}$ is the Weyl basis of $\lie$, the Lie algebra of $G$.	

Case 2: $\Sigma-\Theta=\{\alpha_{i_0}, \alpha_{j_0}\}, ht(\alpha_{i_0})=ht(\alpha_{j_0})=1$ 

Define 

$$\left\{
 \begin{array}{ccc}
 \Delta(1,0)&=&\{\Sigma_i^n m_i\alpha_i\in \Pi^+;m_{i_0}=1 \ \ e \ \ m_{j_0}=0\}\\
 \Delta(0,1)&=&\{\Sigma_i^n m_i\alpha_i\in \Pi^+;m_{i_0}=0 \ \ e\ \ m_{j_0}=1\} \\
 \Delta(1,1)&=&\{\Sigma_i^n m_i\alpha_i\in \Pi^+;m_{i_0}=1 \ \ e\ \ m_{j_0}=1\}\\
  \end{array}
 \right..$$
In this case, 

$$\left\{
 \begin{array}{ccc}
 \m_1&=&\displaystyle\Sigma_{\alpha\in\Delta(1,0)}u_{\alpha}\\
 \m_2&=&\displaystyle\Sigma_{\alpha\in\Delta(0,1)}u_{\alpha}\\ 
 \m_3&=&\displaystyle\Sigma_{\alpha\in\Delta(1,1)}u_{\alpha}\\
  \end{array}
 \right..$$

\subsection*{Flag Manifolds with 4 Isotropic Summands}

When the isotropy representation decomposes into a sum of four pairwise non-equivalent irreducible submodules $\m=\m_1\oplus\m_2\oplus\m_3\oplus\m_4$, we have (\cite{Arvanit4sum}):

\begin{itemize} 
\item [Case 1:] $\Sigma-\Theta=\{\alpha_{i_0}\}$, where $\alpha_{i_0}$ is a simple root of height 4;
\item [Case 2:]$\Sigma-\Theta=\{\alpha_{i_0},\alpha_{j_0}\}$, where $\alpha_{i_0}$ and $\alpha_{j_0}$ are simple roots of height 1 and 2, respectively.   
\end{itemize} 

\begin{theorem}\label{4sum} \emph{(\cite{Arvanit4sum})} Let $G/K$ be a flag manifold such that $\m$ has 4 isotropic summands. Then, $G/K$ is locally isomorphic to one of the following spaces:

\begin{enumerate}

\item $F_4/U(2)\times SU(3)\times SU(2)\times U(1)$;
\item $E_7/SU(4)\times SU(3)\times SU(2)\times U(1)$;
\item $E_8/SO(10)\times SU(3)\times U(1)$;
\item $E_8/SU(7)\times SU(2)\times U(1)$;
\item $SO(2l+1)/SO(2l-3)\times U(1)\times U(1)$;
\item $Sp(l)/U(1)\times U(p)\times U(l-p), \ \ 1\leq p\leq l-1$;
\item $SO(2l)/SO(2l-4)\times U(1)\times U(1)$;
\item $SO(2l)/U(p)\times U(l-p), \ \ 1\leq p\leq l-2$;
\item $E_6/SU(5)\times U(1)\times U(1)$;
\item $E_7/SO(10)\times U(1)\times U(1)$.
\end{enumerate}
The flag manifolds in items \emph{(1)} to \emph{(4)} fall into case 1, while the others fall into case 2.

\end{theorem}

We have the following description (see \cite{Arvanit4sum}) for irreducible submodules of $\m$:

Case 1: $\Sigma-\Theta=\{\alpha_{i_0}\}, ht(\alpha_{i_0})=4$

Define $\Delta(k)=\{\Sigma_i^n m_i\alpha_i\in \Pi^+;m_{i_0}=k\}$, for each integer $1\leq k\leq 4$. The irreducible submodules of $\m$ have the form $$\m_k=\displaystyle\Sigma_{\alpha\in\Delta(k)}u_{\alpha}, u_{\alpha} =\text{span}_{\R}\{X_{\alpha}-X_{-\alpha}, i(X_{\alpha}+X_{-\alpha})\},$$
where $\{X_{\alpha}\}_{\alpha\in\Pi}$ is the Weyl basis of $\lie$, the Lie algebra of $G$.	

Case 2: $\Sigma-\Theta=\{\alpha_{i_0}, \alpha_{j_0}\}, ht(\alpha_{i_0})=1, ht(\alpha_{j_0})=2$ 

Define

$$\left\{
 \begin{array}{ccc}
 \Delta(1,0)&=&\{\Sigma_i^n m_i\alpha_i\in \Pi^+;m_{i_0}=1 \ \ e \ \ m_{j_0}=0\}\\
 \Delta(0,1)&=&\{\Sigma_i^n m_i\alpha_i\in \Pi^+;m_{i_0}=0 \ \ e\ \ m_{j_0}=1\} \\
 \Delta(1,1)&=&\{\Sigma_i^n m_i\alpha_i\in \Pi^+;m_{i_0}=1 \ \ e\ \ m_{j_0}=1\}\\
 \Delta(1,2)&=&\{\Sigma_i^n m_i\alpha_i\in \Pi^+;m_{i_0}=1 \ \ e\ \ m_{j_0}=2\}\\
  \end{array}
 \right..$$
Hence, 

$$\left\{
 \begin{array}{ccc}
 \m_1&=&\displaystyle\Sigma_{\alpha\in\Delta(1,0)}u_{\alpha}\\
 \m_2&=&\displaystyle\Sigma_{\alpha\in\Delta(0,1)}u_{\alpha}\\ 
 \m_3&=&\displaystyle\Sigma_{\alpha\in\Delta(1,1)}u_{\alpha}\\
 \m_4&=&\displaystyle\Sigma_{\alpha\in\Delta(1,2)}u_{\alpha}\\ 
 
  \end{array} 
 \right..$$
\section{Lie Theoretic Approach for the Extremality Condition}


The set of invariant Kähler metrics on $M=G/K$ (with respect to $J$) with the same volume of the Kähler-Einstein metric $\bar{g}$ can be nicely parametrized in the open {\it Weyl chamber} whose construction we will introduce below.

Let $B$ be the Cartan-Killing of the Lie algebra $\lie$ of $G$ and $\mathfrak{k}$ the Lie algebra of $K$. Take a $B$-orthogonal decomposition $\lie=\mathfrak{k}\oplus\mathfrak{m}$, where $\mathfrak{m}$ is the isotropy representation of $K$. Fixed a Cartan subalgebra $\mathfrak{h}$ of $\lie$, denote by $\Pi$ the root system of $\lie^{\mathbb{C}}$ with respect to $\mathfrak{h}^{\mathbb{C}}$ and by $\Sigma$ the system of simple roots.

Hence, 

$$\mathfrak{k}^{\mathbb{C}}=\mathfrak{h}^{\mathbb{C}}\oplus\displaystyle\bigoplus_{\alpha\in \Pi_{\mathfrak{h}}}\lie_{\alpha}, \ \ \mathfrak{m}^{\mathbb{C}}=\displaystyle\bigoplus_{\alpha\in\Pi_{\mathfrak{m}}}\lie_{\alpha},$$
where $\Pi_{\mathfrak{h}}=\left\langle \Theta\right\rangle\cap\Pi$, with $\Theta\subset\Sigma$ and $\Pi_{\mathfrak{m}}=\Pi\backslash \Pi_{\mathfrak{h}}$.  

Let $\mathfrak{c}\subseteq\mathfrak{h}$ be the center of $\mathfrak{k}$. The roots in $\Pi_{\mathfrak{h}}$ are characterized by the fact that they vanish on $\mathfrak{c}$. In our context, we have a relationship between the second Betti number of $G/K$ and $\mathfrak{c}$, namely, $b_2(G/K)=\dim\mathfrak{c}$. 

\begin{remark} We must consider the following (see \cite{aleksev}):

\begin{flushleft}

1) The $G$-invariant complex structures are in bijective correspondence with the invariant orderings of $\Pi_{\mathfrak{m}}$ subsets $\Pi_{\mathfrak{m}}^+\subset\Pi_{\mathfrak{m}}$ such that $\Pi_{\mathfrak{m}}$ is the disjoint union $\Pi_{\mathfrak{m}}=\Pi_{\mathfrak{m}}^+\cup (-\Pi_{\mathfrak{m}}^+)$ and $$(\Pi_{\mathfrak{h}}+\Pi_{\mathfrak{m}}^+)\cap\Pi\subset \Pi_{\mathfrak{m}}^+, \, (\Pi_{\mathfrak{m}}^++\Pi_{\mathfrak{m}}^+)\cap\Pi\subset \Pi_{\mathfrak{m}}^+. $$
\end{flushleft}

\begin{flushleft}

2) Invariant orderings are then in one-to-one correspondence with {\it Weyl chambers} in the center $\mathfrak{c}$, that is, connected components of the set $\mathfrak{c}\backslash\bigcup_{\alpha\in \Pi_{\mathfrak{m}} } \ker (\alpha|_{\mathfrak{c}})$.

\end{flushleft}

\begin{flushleft}

3) Fixed an invariant complex structure $J$ on $M=G/K$ and, therefore, a Weyl chamber in $\mathfrak{c}$, it is possible to parametrize the set of $G$-invariant Kähler metrics which are Hermitian w.r.t. $J$.
\end{flushleft}

\begin{flushleft}

4) The $G$-invariant symplectic structures, i.e., $G$-invariant non-degenerate closed 2-forms, are in one-to-one correspondence with the elements in the Weyl chamber in $\mathfrak{c}$.
\end{flushleft}

Indeed, if $\omega\in\Lambda^2(\mathfrak{m})$ is a symplectic form, there exists  $\xi$ in some Weyl chamber $C$ in $\mathfrak{c}$ such that $$\omega(X,Y)=B(\text{ad}_{\xi}X,Y), \, X,Y\in\mathfrak{m}.$$

In addition, $\omega$ above defined is the Kähler form of a Kähler metric $\bar{g}$ w.r.t. the complex structure $J$, that is, $\bar{g}:=\omega(\cdot,J\cdot)$ defines a Kähler metric if and only if $\xi\in C$.

\end{remark}

Let us consider the functional $$\delta_{\m}=\displaystyle\sum_{\alpha\in \Pi^+_{\m}}\alpha.$$ Its $B$-dual lies in $i\mathfrak{c}$ while $\eta_{\mathfrak{m}}:=-i\hat{\delta}_{\mathfrak{m}}\in C$, where $\hat{\delta}$ is the $B$-dual of $\delta_{\m}$ and $C$ is the fixed Weyl chamber. 

The element $\eta_{\mathfrak{m}}\in C$ defines an invariant Kähler metric, which is the unique invariant Kähler-Einstein metric $\bar{g}$ such that $Ric(\bar{g})=\bar{g}$ (see \cite{bordemann}).

We are interested in the functional $\lambda_1^{IK}:\mathcal{K}_0\longrightarrow \mathbb{R}$, the first eigenvalue, defined on the hypersurface $\mathcal{K}_0\subset C$ given by those points in the Weyl chamber $C$ which corresponds to invariant Kähler metrics on the partial flag manifolds with two, three and four isotropic summands above mentioned, with the same volume as the Kähler-Einstein metric $\bar{g}.$ We will apply a criterion based on $\delta_{\m}$ in order to determine when $\bar{g}$ is $\lambda_1^{IK}$-extremal on such manifolds.

\begin{definition} The Kähler-Einstein metric $\bar{g}$ is \emph{$\lambda_1^{IK}$-extremal} if for each analytic curve $g_t$ in $\mathcal{K}$, $t\in (-\epsilon,\epsilon)$, with $g_0=\bar{g}$, we have $$\frac{d}{dt}|_{t=0^+}\lambda_1(g_t)\leq 0\leq \frac{d}{dt}|_{t=0^-}\lambda_1(g_t).$$ 
\end{definition}



\subsection{Main Results}\label{Section4}


For each analytic curve $g_t$, we define $h:=\frac{d}{dt}|_{t=0}g_t$, which is a $G$-invariant symmetric tensor. Hence, $\left\langle h,\bar{g}\right\rangle$ is constant over $M$. Since $\text{vol}(g_t)=\text{vol}(\bar{g})$, we have $\int_M\left\langle h,\bar{g}\right\rangle d\mu_{\bar{g}}=0$ and $\left\langle h,\bar{g}\right\rangle=0$.

We can identify the set of all symmetric tensors tangent to $g_t$ with the Lie subalgebra $\mathfrak{c}$ and those which correspond to variations such that $\textrm{vol}(g_t)=\textrm{vol}(\bar{g})$, with the hyperplane \linebreak $Y:=\{h\in \mathfrak{c};  \left\langle h,\bar{g}\right\rangle=0\}=\textrm{Ker}(\textrm{Tr})$, where $\textrm{Tr}(h)=\left\langle h,\bar{g}\right\rangle$.

We will take over the fact that if $\bar{g}$ is the Kähler-Einstein metric on $M$, then $\lambda_1(\bar{g})=2$ and the associated eigenspace $E_1$ is isomorphic to the space of Killing fields on $M$ $(E_1\ni u \mapsto J(\textrm{grad}u))$ (\cite{Kobayashi}).

Let $G$ be a compact simple Lie group. from the discussion above, the $G$-invariant inner product, induced by the space $L^2(\bar{g})$ on $E_1$ is given by $r\cdot B$, for some positive constant $r$, so that the base $u_1,\ldots,u_m$ de $E_1$ is $B$-orthogonal too.  

We can also conclude that for each $h\in\mathfrak{c}$ the quadratic form $Q_h$ on $E_1$, given by $$Q_h(u)=\int_M \left\langle du\otimes du + \frac{1}{4}\Delta_{\bar{g}}u^2\cdot \bar{g},h\right\rangle d\mu_{\bar{g}},$$ reduces to
$$Q_h(u)=\int_M \left\langle du\otimes du,h\right\rangle d\mu_{\bar{g}}.$$

Therefore, $Q_h$ is $G$-invariant and there is a functional $\phi\in\mathfrak{c}^{\ast}$ such that $$Q_h(u,u)=\phi(h)\cdot B(u,u).$$ It follows the criterion bellow.

\begin{lemma} \label{trace} \emph{(\cite{Panelli})} Let $G$ be a compact simple Lie group, which coincides locally with the isometry group of $(M,\bar{g})$. Then, the Kähler-Einstein metric $\bar{g}$ is \emph{$\lambda_1^{IK}$-extremal} if and only if there is a constant $c\in \R$ such that \emph{$\phi=c\cdot\textrm{Tr}$}. 
\end{lemma}

\begin{remark} The number of simple roots in $\Sigma-\Theta$, the set of complementary roots, determines the dimension of $\mathfrak{c}\subset\mathfrak{h}\subset\mathfrak{k}$ defined above. From \ref{trace}, if $\textrm{dim}\mathfrak{c}=1$, the Kähler-Einstein metric on the associated flag manifold $G/K$ must be $\lambda_1^{IK}-$extremal.
\end{remark}

We can now enunciate in the following theorem the first classification result of $\lambda_1^{IK}-$extremality for partial flag manifolds in this work. Its proof is an immediate consequence of \ref{trace} as well as the above remark.

\begin{theorem}
The Kähler-Einstein metrics on partial flag manifolds whose the set $\Sigma-\Theta$ of complementary roots contains precisely one simple root are $\lambda_1^{IK}$-extremal.
\end{theorem} 

\begin{corollary}
The respective Kähler-Einstein metrics on partial flag manifolds with two isotropic summands in \ref{2sum}, on the partial flag manifolds with three isotropic summands from items 1 to 7 in \ref{3sum} and on the partial flag manifolds with four isotropic summands from items 1 to 4 in \ref{4sum} are $\lambda_1^{IK}$-extremal.
\end{corollary}

We will apply a criterion, due to Panelli and Podestà (see \cite{Panelli}), which answers the question when the metric $\bar{g}$ on a generalized flag manifold is $\lambda_1^{IK}$-extremal. Note that the following theorem is a reformulation of the previous lemma in terms of algebraic data of the Lie algebra $\mathfrak{g}$ of $G$ and makes sense for $\dim\mathfrak{c}\geq 2$, that is, for $b_2(M)\geq 2,$	where $b_2(M)$ is the second Betti number of the space $M$. 

\begin{theorem} [\cite{Panelli}] Let $G$ be a compact and simple Lie group and let $M=G/K$ a generalized flag manifold endowed with a $G$-invariant complex structure $J$. Let $\bar{g}$ be a Kähler-Einstein metric with $Ric_{\bar{g}}=\bar{g}$ and suppose that $G$ coincides locally with the isometry group of $\bar{g}$. The metric $\bar{g}$ is \emph{$\lambda_1^{KI}$-extremal} if and only if 

\begin{equation}
\delta_{\m}|_{\mathfrak{c}}=\dfrac{\left\|\delta_{\m}\right\|^2}{\dim_{\mathbb{C}}M}\cdot \displaystyle\sum_{\alpha\in \Pi^+_{\m}}\dfrac{\alpha|_{\mathfrak{c}}}{B(\alpha,\delta_{\m})}. \label{criterion}
\end{equation}

\end{theorem}

As a consequence, one has in the same paper, among other results, the following characterization for {\it maximal flag manifolds}.

\begin{theorem} [\cite{Panelli}] Let $G$ be a compact classic Lie group. If $(G/T,J)$ is a maximal flag manifold $(\textrm{dim T}\geq2)$, $\bar{g}$ is $\lambda_1^{IK}$-extremal if and only if $G=SU(3)$.
\end{theorem}

\begin{remark} Let $M=G/T$ be a maximal flag manifold of a compact simple Lie group $G$. The isotropy representation of $M$ has a nice description. Indeed, it can be decomposed into a discrete sum of $2$-dimensional pairwise non-equivalent irreducible $T$-submodules $\m_{\alpha}$ as follows: $$\m=\sum_{\alpha\in \Pi^+}\m_{\alpha}.$$The number of these submodules is equal to the cardinality of $\Pi^+$, the set of positive roots of $\mathfrak{g}$. However, in partial cases, this decomposition must be determined in a particular way for each space.
\end{remark}

To convert the condition \ref{criterion} in a computable formula, we will introduce the notion of $T$-roots and its relationship with the decomposition of $\m$ into isotropic summands.  

Let $B$ be the Cartan-Killing of the Lie algebra $\lie$ of $G$ and $\mathfrak{k}$ the Lie algebra of $K$. Take a $B$-orthogonal decomposition $\lie=\mathfrak{k}\oplus\mathfrak{m}$, where $\mathfrak{m}$ is the isotropy representation of $K$. Fixed a Cartan subalgebra $\mathfrak{h}$ of $\lie$, denote by $\Pi$ the root system of $\lie^{\mathbb{C}}$ with respect to $\mathfrak{h}^{\mathbb{C}}$ and by $\Sigma$ the system of simple roots. Define the restriction map $p:(i\mathfrak{h})^*\longrightarrow (i\mathfrak{c})^*$, given by $p(\alpha)=\alpha|_{\mathfrak{c}}$; the elements of the set $\Pi_T=p(\Pi_{\m})$ are called $T$-roots. It is known (see \cite{alektroot}) that there exists a 1-1 correspondence between $T$-roots $\rho$ and irreducible submodules $\m_{\rho}$ of $\m^{\mathbb{C}}$, namely

$$\Pi_T\ni\rho\leftrightarrow\m_{\rho}=\sum_{p(\alpha)=\rho}\lie_{\alpha},$$
and therefore $$\m^{\mathbb{C}}=\sum_{\rho\in\Pi_T}\m_{\rho}.$$

We can also describe the elements of $\Pi_T$ explicitly as follows. Fixed \linebreak $\Sigma=\{\alpha_1,\alpha_2,\ldots,\alpha_m,\phi_1,\phi_2,\ldots,\phi_k\}$, where $\Theta=\{\phi_1,\phi_2,\ldots,\phi_k\}$, denote by $\{\pi_1,\pi_2,\ldots,\pi_m\}$ the set of fundamental weights associated to $\alpha_1,\alpha_2,\ldots,\alpha_m$, i.e.,   
$$\dfrac{2(\pi_j,\alpha_i)}{(\alpha_i,\alpha_i)}=\delta_{ij}, \quad (\pi_j,\phi_i)=0,$$ where $(\cdot,\cdot)$ is the restriction of the Cartan-Killing form $B$ to $\h^*$. The fundamental weights form a basis to $\mathfrak{c}^*$, so that the $T$-root $\rho=p(\alpha)=\alpha|_{\mathfrak{c}}$ associated to a root $\alpha\in\Pi_{\m},$ is given by 
$$p(\alpha)=\sum_{j=1}^{m}\dfrac{(2\alpha,\alpha_j)}{(\alpha_j,\alpha_j)}\pi_j.$$

Denoting by $\Pi_T^+=\{\rho_1,\rho_2,\ldots,\rho_l\}$ the set of $T$-roots which are image of positive roots in $\Pi^+$, we have that each $\rho_j$ corresponds to the submodule $\m_j$ of complex dimension $m_j$. Then $\m^{\mathbb{C}}=\bigoplus_{i=1}^{l}\m_j\oplus\bar{\m_j}$. For each $\rho\in\Pi_T^+ $, we alson have that the number $B(\alpha,\delta_{\m})$ with $p(\alpha)=\rho$ does not depends on $\alpha$ with $p(\alpha)=\rho$ and we denote it by $B(\rho,\delta_{\m})$. It follows that the condition \ref{criterion} can be expressed as 
\begin{equation}
\displaystyle \sum_{i=1}^l\left(\dfrac{\mu}{\beta_j}-1\right)m_j\rho_j=0, \label{formula}
\end{equation}
where $\beta_j:=B(\rho_j,\delta_{\m})$, $\mu=\left\|\delta_{\m}\right\|^2/\textrm{dim}_{\mathbb{C}}M,$ and $\rho_j$ is the $T$-root associated to the isotropic summand $\m_j$, of dimension $m_j$. 

We are able now to classify completely the $\lambda_1^{IK}$-extremality for Kähler-Einstein metrics $\bar{g}$ on partial flag manifolds with three and four isotropic summands by means formula \ref{formula}. We will begin this classification from the special case represented by the family of partial flag manifolds with three isotropic summands given by $SU(l+m+n)/S(U(l)\times U(m)\times U(n)), \, l,m,n>0$.

\begin{theorem}
For the family of partial flag manifolds with tree isotropic summands given by \linebreak $SU(l+m+n)/S(U(l)\times U(m)\times U(n))$, we have that the corresponding Kähler-Einstein metric is $\lambda_1^{IK}$-extremal if and only if $l=m=n$.
\end{theorem}

\begin{proof}
Let us consider $SU(l+m+n)/S(U(l)\times U(m)\times U(n))$, $l,m,n>0$, with a complex invariant structure $J$, corresponding to the canonical root system of $\mathfrak{su}(l+m+n)$. In this case, \linebreak $\delta_{\m}=(l+m)\Lambda_l+(m+n)\Lambda_{l+m}$, where $\Lambda_l$ and $\Lambda_{l+m}$ are the fundamental weights associated with the simple roots $\Sigma-\Theta=\{\alpha_l, \alpha_{l+m}\}$. Normalizing the C-K form so that $\left\|\alpha\right\|^2=2,\, \forall \, \alpha\in\Pi$, it follows that

\begin{itemize}
\item $\left\|\delta_{\m}\right\|^2=(l+m)(l+n)(m+n)$
\item $\rho_1=\varepsilon_1-\varepsilon_{l+1}$, $\rho_2=\varepsilon_1-\varepsilon_{l+m+1}$, $\rho_3=\varepsilon_{l+1}-\varepsilon_{l+m+1}$
\item $\beta_1=B(\rho_1,\delta_{\m})=(l+m)$, $\beta_2=B(\rho_2,\delta_{\m})=(2m+l+n)$, $\beta_3=B(\rho_3,\delta_{\m})=(m+n)$.
\end{itemize}

Since $\dim\m_1=\left|\Delta(1,0)\right|=lm$, $\dim\m_2=\left|\Delta(0,1)\right|=ln$ and $\dim\m_3=\left|\Delta(1,1)\right|=mn$, we have that $\dim_{\mathbb{C}}M=\dim\m=mn+ln+ln$. Applying the $\lambda_1^{KI}$ criterion, we obtain the following linear combination

$$\left(\dfrac{\mu}{\beta_1}-1\right) m_1\rho_1 +\left(\dfrac{\mu}{\beta_2}-1\right)m_2\rho_2 +\left(\dfrac{\mu}{\beta_3}-1\right)m_3\rho_3=$$
{\small
$$\left(\dfrac{(l+m)(m+n)(l+n)}{(mn+lm+ln)(l+m)}-1\right)lm(\varepsilon_1-\varepsilon_{l+1})$$ $$+\left(\dfrac{(l+m)(m+n)(l+n)}{(mn+lm+ln)(2m+l+n)}-1\right)ln(\varepsilon_1-\varepsilon_{l+m+1})$$ $$+\left(\dfrac{(l+m)(m+n)(l+n)}{(mn+lm+ln)(m+n)}-1\right)mn(\varepsilon_{l+1}-\varepsilon_{l+m+1})=$$ $$\left[\left(\dfrac{(l+m)(m+n)(l+n)}{(mn+lm+ln)(l+m)}-1\right)lm+\left(\dfrac{(l+m)(m+n)(l+n)}{(mn+lm+ln)(2m+l+n)}-1\right)ln\right]\varepsilon_1+$$ $$\left[\left(-\dfrac{(l+m)(m+n)(l+n)}{(mn+lm+ln)(l+m)}+1\right)lm+\left(\dfrac{(l+m)(m+n)(l+n)}{(mn+lm+ln)(m+n)}-1\right)mn\right]\varepsilon_{l+1}+$$ $$\left[\left(-\dfrac{(l+m)(m+n)(l+n)}{(mn+lm+ln)(m+n)}+1\right)mn+\left(-\dfrac{(l+m)(m+n)(l+n)}{(mn+lm+ln)(2m+l+n)}+1\right)ln\right]\varepsilon_{l+m+1}=0$$ }
{\small
$$\Leftrightarrow\left\{
 \begin{array}{ccc}
 \left(\dfrac{(l+m)(m+n)(l+n)}{(mn+lm+ln)(l+m)}-1\right)lm+\left(\dfrac{(l+m)(m+n)(l+n)}{(mn+lm+ln)(2m+l+n)}-1\right)ln&=&0\\
 \left(-\dfrac{(l+m)(m+n)(l+n)}{(mn+lm+ln)(l+m)}+1\right)lm+\left(\dfrac{(l+m)(m+n)(l+n)}{(mn+lm+ln)(m+n)}-1\right)mn&=&0 \\
 \left(-\dfrac{(l+m)(m+n)(l+n)}{(mn+lm+ln)(m+n)}+1\right)mn+\left(-\dfrac{(l+m)(m+n)(l+n)}{(mn+lm+ln)(2m+l+n)}+1\right)ln&=&0\\
  \end{array}
 \right.,$$}since $\{\varepsilon_1,\varepsilon_{l+1},\varepsilon_{l+m+n+1}\}$ is L.I.. The homogeneous system above, with the help of Mathematica software, provides us solution $\{(l,m,n)\in\mathbb{N}^3;l=m=n\}$. Hence, the Kähler-Einstein metric is $\lambda_1^{IK}-$extremal only if $l=m=n$, which icludes the family $SU(3n)/S(U(n)\times U(n)\times U(n))$, $n\geq1$. 

\end{proof}

In \cite{Panelli} it was proved that the Kähler-Einstein metric on $SU(3n)/S(U(n)\times U(n)\times U(n))$ \linebreak is $\lambda_1^{IK}$-extremal. However, in this same work, it is not proved that, for the more general family \linebreak $SU(l+m+n)/S(U(l)\times U(m)\times U(n))$, $l,m,n>0$ the extremality property of the Kähler-Einstein metric only holds when $l=m=n$, as we have just shown.

To conclude the classification of the $\lambda_1^{IK}$-extremality for partial flag manifolds with three isotropic summands, we enunciate and prove the theorem below, where we study the family of homogeneous spaces represented by $SO(2l)/U(1)\times U(l-1)$, $l\geq 4$ as well as the exceptional flag \linebreak $E_6/U(1)\times U(1)\times SO(8)$.

\begin{theorem}
The respective Kähler-Einstein metrics on the partial flag manifolds with three isotropic summands given by $SO(2l)/U(1)\times U(l-1)$, $l\geq 4$ and $E_6/U(1)\times U(1)\times SO(8)$ are not $\lambda_1^{IK}$-extremal. 
\end{theorem} 

\begin{proof}
For $SO(2l)/U(1)\times U(l-1)$, $l\geq 4$, with a complex invariant structure $J$, corresponding to the canonical root system of $\mathfrak{so}(2l)$, we have $\delta_{\m}=l\Lambda_{l-1}+l\Lambda_l$, where $\Lambda_{l-1}$ and $\Lambda_l$ are the fundamental weights associated with the simple roots $\Sigma-\Theta=\{\alpha_{l-1}, \alpha_l\}$. Then 

\begin{itemize}
\item $\left\|\delta_{\m}\right\|^2=\frac{l^2}{l+2}$
\item $\rho_1=\varepsilon_1-\varepsilon_{l}$, $\rho_2=\varepsilon_1+\varepsilon_l$, $\rho_3=\varepsilon_1+\varepsilon_{l-1}$
\item $\beta_1=B(\rho_1,\delta_{\m})=l$, $\beta_2=B(\rho_2,\delta_{\m})=l$, $\beta_3=B(\rho_3,\delta_{\m})=2l$.
\end{itemize}
Moreover, since $\dim\m_1=\left|\Delta(1,0)\right|=l-1$, $\dim\m_2=\left|\Delta(0,1)\right|=l-1$ e $\dim\m_3=\left|\Delta(1,1)\right|=\dfrac{(l-1)(l-2)}{2}$, one has $\dim_{\mathbb{C}}M=\dim\m=\dfrac{(l-1)(l+2)}{2}$.  Applying fórmula \ref{formula}, we obtain the linear combination
$$\left(\dfrac{\mu}{\beta_1}-1\right) m_1\rho_1 +\left(\dfrac{\mu}{\beta_2}-1\right)m_2\rho_2 +\left(\dfrac{\mu}{\beta_3}-1\right)m_3\rho_3$$
$$= \left(\dfrac{l^2}{l(l+2)}-1\right)(l-1)(\varepsilon_1-\varepsilon_l)+\left(\dfrac{l^2}{l(l+2)}-1\right)(l-1)(\varepsilon_1+\varepsilon_l)$$ $$+\left(\dfrac{l^2}{2l(l+2)}-1\right)\dfrac{(l-1)(l-2)}{2}(\varepsilon_1+\varepsilon_{l-1})$$ $$=\dfrac{(l-1)(l^2+2l+8)}{4(l+2)}\varepsilon_1+\dfrac{(l-1)(l-2)(l+4)}{4(l+2)}\varepsilon_{l-1}=0$$

$$\Leftrightarrow\left\{
 \begin{array}{ccc}
 (l-1)(l^2+2l+8)&=&0\\
 (l-1)(l-2)(l+4)&=&0 \\
   \end{array}
 \right.,$$
Since $\{\varepsilon_1,\varepsilon_{l-1},\varepsilon_l\}$ is L.I., the homogeneous system above provides us solution $l=1$, which implies that the Kähler-Einstein metric on $SO(2l)/U(1)\times U(l-1)$, $l\geq 4,$ is not $\lambda_1^{IK}-$extremal.

Now considering the flag manifold associated with the exceptional group $E_6$, given by\linebreak $E_6/U(1)\times U(1)\times SO(8)$, with a complex invariant structure $J$, corresponding to the canonical root system of $\mathfrak{e}_6$ and a subset of complementary roots $\Sigma-\Theta=\{\alpha_1, \alpha_6\}$, it folows that $\delta_{\m}=4\Lambda_1+4\Lambda_6$, where $\Lambda_1$ and $\Lambda_6$ are the fundamental weights associated with the simple roots $\alpha_1$ and $\alpha_6$. In this way, 

 \begin{itemize}
\item $\left\|\delta_{\m}\right\|^2=4$
\item $\rho_1=-(\varepsilon_3+\varepsilon_4+\varepsilon_6)$, $\rho_2=-(\varepsilon_2+\varepsilon_4+\varepsilon_7)$, $\rho_3=-(\varepsilon_3+\varepsilon_4+\varepsilon_7)$
\item $\beta_1=B(\rho_1,\delta_{\m})=\beta_2=B(\rho_2,\delta_{\m})=4$, $\beta_3=B(\rho_3,\delta_{\m})=8$.
\end{itemize} 

Since $\dim\m_1=\dim\m_2=\dim\m_3=8$, we have $\dim_{\mathbb{C}}M=\dim\m=24$. Applying \ref{formula}, we obtain

$$-(\mu-8)\rho_1-(2\mu-8)\rho_2-(\mu-8)\rho_3,$$ where $\mu=\left\|\delta_{\m}\right\|^2/\dim_{\mathbb{C}}M=1/6$, which is nonzero, since the $T$-roots $\rho_1, \rho_2$ and $\rho_3$ are L.I., i.e., from \ref{formula} we have that the corresponding Kähler-Einstein is not $\lambda_1^{IK}$-extremal.

\end{proof}

From now on, we will deal with the partial flag manifolds with four isotropic summands from items \emph{(5)} to \emph{(10)} in \ref{4sum}.

\begin{theorem}
The respective Kähler-Einstein metrics on the partial flag manifolds with four isotropic summands given by $SO(2l+1)/SO(2l-3)\times U(1)\times U(1), Sp(l)/U(1)\times U(p)\times U(l-p),$ \linebreak $1\leq p\leq l-1, SO(2l)/SO(2l-4)\times U(1)\times U(1), SO(2l)/U(p)\times U(l-p), \ \ 1\leq p\leq l-2,$ $E_6/SU(5)\times U(1)\times U(1), E_7/SO(10)\times U(1)\times U(1)$ are not $\lambda_1^{IK}$-extremal. 
\end{theorem}

\begin{proof} 

Let us consider $SO(2l+1)/SO(2l-3)\times U(1)\times U(1)$, with a complex invariant structure $J$, corresponding to the canonical root system of $\mathfrak{so}(2l+1)$. Thus, $\delta_{\m}=2\Lambda_1+(2l-3)\Lambda_2$, where $\Lambda_1$ and $\Lambda_2$ are the fundamental weights associated with the simple roots $\Sigma-\Theta=\{\alpha_1, \alpha_2\}$. It follows that

\begin{itemize}
\item $\left\|\delta_{\m}\right\|^2=(2l-1)^2$.
\item $\rho_1=\varepsilon_1-\varepsilon_2$, $\rho_2=\varepsilon_2+\varepsilon_3$, $\rho_3=\varepsilon_1+\varepsilon_3, \rho_4=\varepsilon_1+\varepsilon_2$.
\item $\beta_1=B(\rho_1,\delta_{\m})=2$, $\beta_2=B(\rho_2,\delta_{\m})=2l-3$, $\beta_3=B(\rho_3,\delta_{\m})=2(l-1)$, $\beta_4=B(\rho_4,\delta_{\m})=4(l-1)$. 
\end{itemize}
Since $m_1=\emph{dim}\m_1=\left|\Delta(1,0)\right|=1$, $m_2=\dim\m_2=\left|\Delta(0,1)\right|=2l-3$ e $m_3=\dim\m_3=\left|\Delta(1,1)\right|=2l-3$ e $m_4=\dim\m_4=\left|\Delta(1,2)\right|=1$ we have $\dim_{\mathbb{C}}M=\dim_{\mathbb{C}}\m=4(l-1)$.  Hence, applying \ref{formula}, we have
$$\left(\dfrac{\mu}{\beta_1}-1\right) m_1\rho_1 +\left(\dfrac{\mu}{\beta_2}-1\right)m_2\rho_2 +\left(\dfrac{\mu}{\beta_3}-1\right)m_3\rho_3+\left(\dfrac{\mu}{\beta_4}-1\right)m_4\rho_4=$$
$$\left(\dfrac{(2l-1)^2}{8(l-1)}-1\right)(\varepsilon_1-\varepsilon_2)+\left(\dfrac{(2l-1)^2}{4(l-1)(2l-3)}-1\right)(2l-3)(\varepsilon_2+\varepsilon_3)$$ $$+\left(\dfrac{(2l-1)^2}{8(l-1)^2}-1\right)(2l-3)(\varepsilon_1+\varepsilon_3)+\left(\dfrac{(2l-1)^2}{16(l-1)^2}-1\right)(\varepsilon_1+\varepsilon_2)=$$
$$\left(-\dfrac{(3-2l)^2(2l-1)}{16(l-1)^2}\right)\varepsilon_1+\left(\dfrac{47-122l+100l^2-24l^3}{16(l-1)^2}\right)\varepsilon_2+\dfrac{43-104l+76l^2-16l^3}{8(l-1)^2}\varepsilon_3= 0$$

$$\Leftrightarrow\left\{
 \begin{array}{ccc}
 (3-2l)^2(2l-1)&=&0\\
 47-122l+100l^2-24l^3&=&0 \\
 43-104l+76l^2-16l^3&=&0\\
\end{array}
\right.,
$$
which has no solution in $\mathbb{N}$. Therefore, the Kähler-Einstein metric $\bar{g}$ on \linebreak $SO(2l+1)/SO(2l-3)\times U(1)\times U(1)$ is not $\lambda_1^{IK}-$extremal, for every $l\geq 2$.

Given $Sp(l)/U(1)\times U(p)\times U(l-p), \ \ 1\leq p\leq l-1$, with a complex invariant structure $J$, corresponding to the canonical root system of $\mathfrak{sp}(l)$, we have $\delta_{\m}=l\Lambda_p+(l-p+1)\Lambda_l$, where $\Lambda_p$ and $\Lambda_l$ are the fundamental weights associated with the simple roots $\Sigma-\Theta=\{\alpha_p, \alpha_l\}$. Thus,

\begin{itemize}
\item $\left\|\delta_{\m}\right\|^2=\dfrac{l^2p+2lp(l-p+1)+l(l-p+1)^2}{4(l+1)}$.
\item $\rho_1=\varepsilon_1-\varepsilon_l$, $\rho_2=\varepsilon_{p+1}+\varepsilon_l$, $\rho_3=\varepsilon_1+\varepsilon_l, \rho_4=\varepsilon_1+\varepsilon_p$.
\item $\beta_1=B(\rho_1,\delta_{\m})=l$, $\beta_2=B(\rho_2,\delta_{\m})=l-p+1$, $\beta_3=B(\rho_3,\delta_{\m})=2l+p-1$, $\beta_4=B(\rho_4,\delta_{\m})=3l+p-1$. 
\end{itemize}

Since $m_1=\dim\m_1=\left|\Delta(1,0)\right|=p(l-p)$, $m_2=\dim\m_2=\left|\Delta(0,1)\right|=(l-p)(l-p+1)$ e $m_3=\dim\m_3=\left|\Delta(1,1)\right|=p(l-p)$ and $m_4=\dim\m_4=\left|\Delta(1,2)\right|=p(p+1)$, so that $\dim_{\mathbb{C}}M=\dim_{\mathbb{C}}\m=l(l+1)$.  Applying \ref{formula}, we obtain
$$\left(\dfrac{\mu}{\beta_1}-1\right) m_1\rho_1 +\left(\dfrac{\mu}{\beta_2}-1\right)m_2\rho_2 +\left(\dfrac{\mu}{\beta_3}-1\right)m_3\rho_3+\left(\dfrac{\mu}{\beta_4}-1\right)m_4\rho_4=$$
{\footnotesize $$-\dfrac{p (3 l^5 - (-1 + p)^3 p (1 + p) - 2 l^4 (3 + 2 p) -l (-1 + p)^2 (-1 + 3 p + 3 p^2) + l^3 (-8 + 12 p + 11 p^2) + 
   l^2 (2 - 3 p - 2 p^2 + 3 p^3))}{l (1 + l) (-1 + 2 l + p) (-1 + 3 l + p)}\varepsilon_1$$} $$-\dfrac{(l - p) (3 l^3 + l^2 (5 - 4 p) - 4 (-1 + p) p + 2 l (1 + p))}{l (1 + l) (-1 + 2 l + p)}\varepsilon_l$$ $$+\left(\dfrac{(l - p) (1 + 2 l - p) p}{l+1}\right)\varepsilon_{p+1}$$ $$+\left(-\dfrac{p (1 + p) (-2 + 2 l^2 + p + p^2)}{(1 + l) (-1 + 3 l + p)}\right)\varepsilon_p.$$ This linear combination must be zero if and only if each of the above coefficients is zero, since $\{\varepsilon_1,\varepsilon_l,\varepsilon_{p+1}, \varepsilon_p\}.$ is L.I. Using Mathematica software, we can see that the corresponding homogeneous system has a solution only if $p=0$, which allows us to conclude that the Kähler-Einstein metric on the flag manifold, in this case, is not $\lambda_1^{IK}-$extremal, since $1\leq p\leq l-1$.

For $M=SO(2l)/SO(2l-4)\times U(1)\times U(1), \ \ l\geq 4$, with a complex invariant structure $J$, corresponding to the canonical root system of $\mathfrak{so}(2l)$, we have $\delta_{\m}=2\Lambda_1+2(l-2)\Lambda_2$, where $\Lambda_1$, $\Lambda_2$ are the fundamental weights associated with the simple roots $\Sigma-\Theta=\{\alpha_1, \alpha_2\}$. Accordingly,  

 \begin{itemize} 
\item $\left\|\delta_{\m}\right\|^2=4(5-6l+2l^2)$
\item $\rho_1=\varepsilon_1-\varepsilon_2$, $\rho_2=\varepsilon_2-\varepsilon_3$, $\rho_3=\varepsilon_1+\varepsilon_3$, $\rho_4=\varepsilon_1+\varepsilon_2$.
\item $\beta_1=B(\rho_1,\delta_{\m})=2$, $\beta_2=B(\rho_2,\delta_{\m})=2(l-2)$, $\beta_3=B(\rho_3,\delta_{\m})=2(l-1)$, $\beta_4=B(\rho_4,\delta_{\m})=2(2l-3)$.
\end{itemize} 

Since $m_1=\dim\m_1=\left|\Delta(1,0)\right|=1$, $m_2=\dim\m_2=\left|\Delta(0,1)\right|=2(l-2)$, $m_3=\dim\m_3=\left|\Delta(1,1)\right|=2(l-2)$ e $m_4=\dim\m_4=\left|\Delta(1,2)\right|=1$, one has $\dim_{\mathbb{C}}M=\dim_{\mathbb{C}}\m=2(2l-3)$.  Applaying \ref{formula}, we have

$$\left(\dfrac{\mu}{\beta_1}-1\right) m_1\rho_1 +\left(\dfrac{\mu}{\beta_2}-1\right)m_2\rho_2 +\left(\dfrac{\mu}{\beta_3}-1\right)m_3\rho_3+\left(\dfrac{\mu}{\beta_4}-1\right)m_4\rho_4=$$
$$\dfrac{2 (26 - 57 l + 46 l^2 - 16 l^3 + 2 l^4)}{(3 - 2 l)^2 (-1 + l)}\varepsilon_1+\dfrac{26 - 44 l + 24 l^2 - 4 l^3}{(3 - 2 l)^2}\varepsilon_2+\dfrac{2}{l-1}\varepsilon_3\neq 0,$$
since $\{\varepsilon_1,\varepsilon_2,\varepsilon_3\}$ is L.I. Hence, the métrica Kähler-Einstein metric on $M$ is not $\lambda_1^{IK}-$extremal.

Considering the partial flag manifold associated with the Lie algebra $\mathfrak{so}(2l)$, namely $M=SO(2l)/SO(2l-4)\times U(p)\times U(l-p), \ \ 2\leq p\leq l-2$, with a complex invariant structure $J$, corresponding to the canonical root system of $\mathfrak{so}(2l)$, we have that $\delta_{\m}=l\Lambda_p+2(l+p-1)\Lambda_l$, where $\Lambda_p$, $\Lambda_l$ are the fundamental weights associated with the simple roots $\Sigma-\Theta=\{\alpha_p, \alpha_l\}$. Thus, 

 \begin{itemize} 
\item $\left\|\delta_{\m}\right\|^2=l^2p+2lp(l+p-1)+(l+p-1)^2$
\item $\rho_1=\varepsilon_{p+1}-\varepsilon_{l-1}$, $\rho_2=\varepsilon_1-\varepsilon_{p+1}$, $\rho_3=\varepsilon_1+\varepsilon_{l-1}$, $\rho_4=\varepsilon_1+\varepsilon_2$.
\item $\beta_1=B(\rho_1,\delta_{\m})=2(l+p-1)$, $\beta_2=B(\rho_2,\delta_{\m})=l$, $\beta_3=B(\rho_3,\delta_{\m})=3l+2p-2$, $\beta_4=B(\rho_4,\delta_{\m})=4l+2p-2$.
\end{itemize} 

Since $m_1=\dim\m_1=\left|\Delta(1,0)\right|=p(l-p)$, $m_2=\dim\m_2=\left|\Delta(0,1)\right|=\frac{(l-p)(l-p-1)}{2}$, $m_3=\dim\m_3=\left|\Delta(1,1)\right|=p(l-p)$ e $m_4=\dim\m_4=\left|\Delta(1,2)\right|=\frac{p(p-1)}{2}$, whence $\dim_{\mathbb{C}}M=\dim_{\mathbb{C}}\m=\dfrac{l^2-2p^2+l(2p-1)}{2}$. Then, $\mu=\dfrac{2(l^2p+2lp(l+p-1)+(l+p-1)^2)}{l^2-2p^2+l(2p-1)}$ and therefore
\newpage
$$\left(\dfrac{\mu}{\beta_1}-1\right) m_1\rho_1 +\left(\dfrac{\mu}{\beta_2}-1\right)m_2\rho_2 +\left(\dfrac{\mu}{\beta_3}-1\right)m_3\rho_3+\left(\dfrac{\mu}{\beta_4}-1\right)m_4\rho_4=$$

{\small $$\left(\left(\dfrac{\mu}{2(l+p-1)}-1\right)p(l-p)+\left(-\dfrac{\mu}{l}+1\right)\dfrac{(l-p)(l-p-1)}{2}\right)\varepsilon_{p+1}$$ $$+\left(\left(\dfrac{\mu}{2(l+p-1)}-1\right)p(l-p)+\left(\dfrac{\mu}{3l+2p-2}-1\right)p(l-p)\right)\varepsilon_{l-1}$$ $$+\left(\left(\dfrac{\mu}{l}-1\right)\dfrac{(l-p)(l-p-1)}{2}+\left(\dfrac{\mu}{3l+2p-2}-1\right)p(l-p)+\left(\dfrac{\mu}{4l+2p-2}-1\right)\dfrac{p(p-1)}{2}\right)\varepsilon_1$$ $$+\left(\dfrac{\mu}{4l+2p-2}-1\right)\dfrac{p(p-1)}{2}\varepsilon_2= 0$$}

\small{ 
$$\Leftrightarrow\left\{
 \begin{array}{ccc}
 \left(\dfrac{\mu}{2(l+p-1)}-1\right)p(l-p)+\left(-\dfrac{\mu}{l}+1\right)\dfrac{(l-p)(l-p-1)}{2}&=&0\\
 \left(\dfrac{\mu}{2(l+p-1)}-1\right)p(l-p)+\left(\dfrac{\mu}{3l+2p-2}-1\right)p(l-p)&=&0 \\
 \left(\dfrac{\mu}{l}-1\right)\dfrac{(l-p)(l-p-1)}{2}+\left(\dfrac{\mu}{3l+2p-2}-1\right)p(l-p)+\left(\dfrac{\mu}{4l+2p-2}-1\right)\dfrac{p(p-1)}{2}&=&0\\
\left(\dfrac{\mu}{4l+2p-2}-1\right)\dfrac{p(p-1)}{2}&=&0
\end{array}
\right.,
$$}since $\{\varepsilon_1,\varepsilon_2,\varepsilon_{l-1}, \varepsilon_{p+1}\}$ is L.I. Using the Mathematica software, we can see that the above system has no solution. Therefore, the Kähler-Einstein metric on the flag manifold $M$ is not $\lambda_1^{IK}$-extremal.


In the case of the partial flag manifold $M=E_6/SU(5)\times U(1)\times U(1)$, with a complex invariant structure $J$, corresponding to the canonical root system of $\mathfrak{e}_6$ and with a subset of complementary roots $\Sigma-\Theta=\{\alpha_1, \alpha_2\}$, one has $\delta_{\m}=2\Lambda_1+8\Lambda_2$, where $\Lambda_1$ e $\Lambda_2$ are the fundamental weights associated with the simple roots $\alpha_1$ and $\alpha_2$. Hence, 

 \begin{itemize}
\item $\left\|\delta_{\m}\right\|^2=\frac{568}{13}$
\item $\rho_1=\varepsilon_2-\varepsilon_3$, $\rho_2=-(\varepsilon_2+\varepsilon_4+\varepsilon_7)$, $\rho_3=-(\varepsilon_3+\varepsilon_4+\varepsilon_7)$, $\rho_4=-(\varepsilon_4+\varepsilon_5+\varepsilon_6)$
\item $\beta_1=B(\rho_1,\delta_{\m})=2, B(\rho_2,\delta_{\m})=8$, $\beta_3=B(\rho_3,\delta_{\m})=10$, $\beta_4=B(\rho_4,\delta_{\m})=18$.
\end{itemize} 

Since $\dim\m_1=1, \dim\m_2=\dim\m_3=10$ e $\dim\m_4=5$, we have $\dim_{\mathbb{C}}M=\dim_{\mathbb{C}}\m=26$. Applying \ref{formula}, we get the following linear combination 

$$\left(\dfrac{\mu}{\beta_1}-1\right) m_1\rho_1 +\left(\dfrac{\mu}{\beta_2}-1\right)m_2\rho_2 +\left(\dfrac{\mu}{\beta_3}-1\right)m_3\rho_3+\left(\dfrac{\mu}{\beta_4}-1\right)m_4\rho_4=$$
$$-\frac{309}{13}\varepsilon_2-\frac{709}{13}\varepsilon_3-\frac{769}{13}\varepsilon_4-\frac{1018}{13}\varepsilon_7-\frac{835}{117}\varepsilon_6\neq 0$$

Therefore, the criterion given by formula \ref{formula} assures us that the Kähler-Einstein metric on M is not \linebreak $\lambda_1^{IK}-$extremal.

The flag manifold $M=E_7/SO(10)\times U(1)\times U(1)$, with a complex invariant structure $J$, corresponding to the canonical root system of $\mathfrak{e}_7$ and with a subset of complementary roots $\Sigma-\Theta=\{\alpha_1, \alpha_2\}$, is such that $\delta_{\m}=2\Lambda_1+12 \Lambda_2$, where $\Lambda_1$ e $\Lambda_2$ are the fundamental weights associated with the simple roots $\alpha_1$ and $\alpha_2$. In this way, 

 \begin{itemize}
\item $\left\|\delta_{\m}\right\|^2=\frac{678}{43}$
\item $\rho_1=\varepsilon_2-\varepsilon_3$, $\rho_2=-(\varepsilon_2+\varepsilon_4+\varepsilon_8)$, $\rho_3=-(\varepsilon_3+\varepsilon_4+\varepsilon_8)$, $\rho_4=-(\varepsilon_4+\varepsilon_6+\varepsilon_7)$
\item $\beta_1=B(\rho_1,\delta_{\m})=2, B(\rho_2,\delta_{\m})=12$, $\beta_3=B(\rho_3,\delta_{\m})=14$, $\beta_4=B(\rho_4,\delta_{\m})=26$.
\end{itemize} 

SInce $\dim\m_1=1, \dim\m_2=\dim\m_3=16$ and $\dim\m_4=10$, we have $\dim_{\mathbb{C}}M=\dim_{\mathbb{C}}\m=43$. Applying \ref{formula}, we obtain 

$$\left(\dfrac{\mu}{\beta_1}-1\right) m_1\rho_1 +\left(\dfrac{\mu}{\beta_2}-1\right)m_2\rho_2 +\left(\dfrac{\mu}{\beta_3}-1\right)m_3\rho_3+\left(\dfrac{\mu}{\beta_4}-1\right)m_4\rho_4=$$
$$\frac{40}{3913}\left[182\varepsilon_2-871\varepsilon_3-304\varepsilon_4+385\varepsilon_6+385\varepsilon_7-689\varepsilon_8\right]\neq 0$$

Hence, from \ref{formula}, we can conclude that the corresponding Kähler-Einstein metric on $M$ is not $\lambda_1^{IK}$-extremal.

\end{proof}



\bibliography
{unsrt}  


\end{document}